\documentclass[final,reqno]{amsart}

\usepackage{graphicx}
\usepackage{amsmath,amsthm,amsfonts,amscd,amssymb}
\usepackage[color,notref,notcite]{showkeys}
\usepackage{url}
\usepackage{subeqnarray}
\usepackage[color,notref,notcite]{showkeys} % for displaying labels in the margins
\definecolor{labelkey}{rgb}{0.6,0,1}
\usepackage{enumitem}%citesort
\usepackage{algorithm}
\usepackage{algpseudocode}
\usepackage{citesort}

\theoremstyle{plain}
\newtheorem{theorem}{Theorem}[section]

\newtheorem{lemma}[theorem]{Lemma}

\newtheorem{assumptions}[theorem]{Assumptions}

\theoremstyle{definition}
\newtheorem{definition}[theorem]{Definition}

\def\bhyp#1{\begin{equation}\label{#1}\begin{array}{c}}
\def\ehyp{\end{array}\end{equation}}

\newcounter{cst}

\theoremstyle{remark}
\newtheorem{remark}[theorem]{Remark}

\numberwithin{equation}{section}
\numberwithin{figure}{section}

\newcommand{\RR}{{\mathbb R}}
\newcommand{\NN}{{\mathbb N}}

\def\O{\Omega}
\def\dsp{\displaystyle}

\def\disc{{\mathcal D}}
\def\mesh{{\mathcal M}}

\def\dr{\partial}
\newcommand{\bw}{{\bf w}}

\def\cT{\mathbb{T}}

\newcommand{\x}{\pmb{x}}

\newif\ifcorr\corrtrue
\usepackage[normalem]{ulem}
\definecolor{violet}{rgb}{0.580,0.,0.827}

\def\bo{{\boldsymbol \omega}}

\def\bpsi{{\boldsymbol \psi}}
\def\bphi{{\boldsymbol \varphi}}

\newcommand{\ud}{\, \mathrm{d}} % for getting the roman script d in integrals
\def\div{\mathop{\rm div}}

\title[Error Estimates of Approximation of SRDE{s}]{General Error Estimates of Non Conforming Approximation of System of Reaction-Diffusion Equations}

\author{Yahya Alnashri}
\address[Yahya Alnashri]{Department of Mathematics, Al-Qunfudah University College, Umm Al-Qura University, Saudi Arabia}
\email{yanashri@uqu.edu.sa}

\date{\today}

\begin{document}
\newcommand{\subscript}[2]{$#1 _ #2$}
% --- ABSTRACT ---

\begin{abstract}
This paper aims to establish a first general error estimate for numerical approximations of the system of reaction-diffusion equations (SRDEs), using reasonable regularity assumptions on the exact solutions. We employ the gradient discretisation method (GDM) to discretise the system and prove the existence and uniqueness of the approximate solutions. The analysis provided here is not limited to specific reaction functions, and it is applicable to all conforming and non-conforming schemes fitting within the GDM framework. As an application, we present numerical results based on a finite volume method.

%The rates of convergence in $L_2$ norm for such solutions and their gradients are at least one as expected in low-order methods such as the HMM method.
\end{abstract}

\maketitle
%%%%%%%%%%%%%%%%%%%%%%%%%%%%%%%%%%%%%%%%%%%%%%%%%%%%%%%%%%%%%%%%%%%%%%

%---- Introduction --------
\section{Introduction}

A variety of issues in physics, biology, chemistry, and elasticity are described using a system of reaction-diffusion equations (SRDEs). For instance, in \cite{2,4}, it simulates the interaction between chemical substances. In biological phenomena, a reaction-diffusion model is proposed to mimic the spread of infectious diseases and population growth \cite{16} and generate stem cells \cite{1}. FitzHugh-Nagumo equation \cite{17} and the Brusselator model \cite{5} are specific types of reaction-diffusion systems used to study electrical and physical problems.
%-----------------------

Letting $\O\subset \RR^d$ ($d\geq 1$) be a polygonal bounded domain and $T>0$, we consider in this work the following system of partial differential equations coupled with generic reaction terms and anisotropic tensors:
\begin{subequations}\label{eq-rm}
\begin{align}
\partial_t \bar u(\x,t) -\div(\Lambda_1(\x)\nabla \bar u(\x,t))&=F_1(\bar u,\bar v), \quad (\x,t) \in \O\times (0,T),\label{rm-strong1}\\
\partial_t \bar v(\x,t) -\div(\Lambda_2(\x)\nabla \bar v(\x,t))&=F_2(\bar u,\bar v), \quad (\x,t) \in \O\times (0,T),\label{rm-strong2}\\
(\bar u(\x,t),\bar v(\x,t))&= (0,0), \quad (\x,t) \in\partial\O\times (0,T), \label{rm-strong3}\\
(\bar u(\x,0),\bar v(\x,0))&=(u_{\rm ini}(\x),v_{\rm ini}(\x)), \quad \x \in \O. \label{rm-strong5}
\end{align}
\end{subequations}
%--------------------------------------

With particular choices of reaction terms $F_1$ and $F_2$, Equations \eqref{rm-strong1}--\eqref{rm-strong5} cover different mathematical models, such as FitzHugh-Nagumo equation \cite{15,28} and the Brusselator problem \cite{27}. For simplicity, we focus on Dirichlet homogeneous boundary conditions, though our analysis is adaptable to other boundary types conditions.
%-----------------------------------------

Mathematical studies corresponding to the existence and uniqueness of the solutions to SRDEs are available in different references \cite{18,19,20}. For the numerical aspect, reaction-diffusion systems have been approximated by several numerical methods. We cite finite difference \cite{8,9}, the spectral method \cite{10}, finte element method \cite{11,12}. A lumped surface ﬁnite element method is applied to such a system in \cite{7}, and it is shown that its convergence is a quadratic respect to mesh size under strong regularity on the continuous solutions and reaction terms. Extensions of the non-linear Galerkin method to approximate the considered model can be found in \cite{13}. To our best knowledge, most of the previous schemes are restricted to rectangular meshes and an isotropic diffusion ($\Lambda_1=\Lambda_2=I_d$). For polygonal meshes, we cite \cite{15} applies a finite volume method to FitzHugh-Nagumo equations, a specific type of SRDEs. In \cite{14}, the SRDEs are also discretised by a finite volume scheme, but no convergence analysis is carried out. We provide in \cite{6} a generic numerical analysis of Problem \eqref{eq-rm} with $\Lambda_1=\Lambda_2=I_d$, and only prove the convergence of the scheme up to a subsequence of the discrete solutions towards a weak continuous solution, but no convergence orders are derived.
%--------------------

While benchmarking is typically employed to evaluate the quality of numerical schemes, it is limited to certain specific cases. This study seeks to determine the convergence order of generic approximations of the system \eqref{eq-rm} according to relevant discrete norms. The error estimate provides a robust approach to assess the efficiency of numerical methods. To the best of our knowledge, this is the first work to present general error estimates for the approximation by non-conforming methods of the anisotropic SRDM. To achieve our results, we approximate the studied model in space using the Gradient Discretisation Method detailed in\cite{21}, an abstract framework that encompasses large families of both conforming and non-conforming schemes.
%--------------------------

The organisation of this paper is structured as follows. In Section \ref{Sec-2}, we recall the weak problem and the definition of the gradient discretisation method. In Section \ref{sec-3}, we state and prove the main results (the error estimates). In addition to employing a similar technique as that in \cite{22}, the proof necessitates a novel approach to handling the system of equations and the general non-linear reaction functions. Section \ref{sec-num} focuses on a numerical test that validates our theoretical convergence order.

% --- Continous and Discrete Problems ---
\section{Weak Formulation and Discretisation}\label{Sec-2}
Let us start by introducing the following assumptions on data.

%--------------------
\begin{assumptions}\label{assump-rm}
We assume that the model data satisfies the following conditions:
\begin{itemize}
\item $\O \subset \RR^d$ ($d=1,2.3$) with a $C^2\mbox{--regular}$ boundary $\dr\O$ and $T>0$,
\item $\Lambda_1,\Lambda_2:\O \to \mathbb A _d(\RR)$ are measurable functions, such that, for a.e. $\Lambda_1(\x)$ and $\Lambda_2(\x)$ are symmetric with eigenvalues in $[\underline\lambda_1,\overline\lambda_1]$ and $[\underline\lambda_2,\overline\lambda_2]$, respectively, where $\mathbb A _d(\RR)$ is the set of $d\times d$ matrices,
\item $(u_{\rm ini},v_{\rm ini})$ are in $H^1(\O) \times H^1(\O)$,
\item the functions $F_1,\; F_2: \RR^2 \to \RR$ are Lipschitz continuous with Lipschitz constants $L_1$ and $L_2$, respectively, where $L:=\max\{L_1,L_2\}$.
\end{itemize}
\end{assumptions}
%------------------------------------------------------------

Under the above assumptions, it is proved in \cite{20} that the problem \eqref{eq-rm} admits a unique pair of weak solutions as follows: there exists $(\bar u,\bar v) \in (L^2(0,T;H_0^1(\O)))^2$ such that $(\dr_t\bar u,\dr_t\bar v) \in (L^2(0,T;H^{-1}(\O)))^2$, and for all a.e. $t\in[0,T]$, the following equalities hold:
\begin{subequations}\label{rm-weak}
\begin{equation}\label{rm-weak1}
\begin{aligned}
\langle \dr_t\bar u(\cdot,t),\varphi(\cdot,t) \rangle  
&+\dsp\int_\O \Lambda_1(\x)\nabla \bar u(\x,t) \cdot \nabla \varphi(\x,t)\ud \x\\
{}&\quad=\dsp\int_\O F_1(\bar u,\bar v)\varphi(\x,t) \ud \x,
\quad \forall \varphi\in L^2(0,T;H_0^1(\O)),
\end{aligned}
\end{equation}
\begin{equation}\label{rm-weak2}
\begin{aligned}
\langle \dr_t\bar v(\cdot,t),\psi(\cdot,t) \rangle  
&+\dsp\int_\O \Lambda_2(\x)\nabla \bar v(\x,t) \cdot \nabla \psi(\x,t)\ud \x\\
{}&\quad=\dsp\int_\O F_2(\bar u,\bar v)\psi(\x,t) \ud \x,
\quad \forall \psi\in L^2(0,T;H_0^1(\O)),
\end{aligned}
\end{equation}
\end{subequations}
%------------------------
where $\langle \cdot,\cdot \rangle$ is the duality product between $H^1(\O)$ and its dual space $H^{-1}(\O)$.

Concerning the spatial space, we can approximate the problem \eqref{rm-weak} by the gradient discretisation method. The first step toward this approximation is introducing generic discrete spaces and operators as in the following definition. 
%------------------------------
\begin{definition}[Generic Discrete Elements]\label{def-gd}
A space-time gradient discretisation $\disc$ is given by $\disc=(X_{\disc,0}, \Pi_\disc, \nabla_\disc, J_\disc)$, where
\begin{itemize}
\item $X_{\disc,0}$ is a finite-dimensional space on $\RR$, associated with the interior and boundary degrees of freedoms and taking into account the zero boundary conditions,
\item $\Pi_\disc : X_{\disc} \to L^2(\O)$ is a linear reconstructed function,
\item $\nabla_\disc : X_{\disc} \to L^2(\O)^d$ is a linear reconstructed gradient and must be defined so that $\| \cdot ||_\disc := || \nabla_\disc \cdot ||_{L^2(\O)^d}$ defines a norm on $X_{\disc,0}$,
\item $J_\disc: L^\infty(\O) \to X_\disc$ is a linear and continuous interpolation operator for the initial conditions. 
\end{itemize}
\end{definition}
%----------------------------------------
The coercivity, the consistency and the dual consistency of the discretisation are respectively controlled by the parameter $C_\disc$, the interpolation error function $S_\disc : H_0^1(\O)\to [0, +\infty)$, and the conformity error function $W_\disc : H_{\rm{div}}(\O)= \{\bo \in L^{2}(\O)^{d}{\;:\;} {\rm div}\bo \in L^2(\O)\} \to [0, +\infty)$; they are defined by    
%-----------------
\begin{equation}\label{CD}
C_\disc =  {\dsp \max_{w \in X_{\disc,0}\setminus\{0\}}\frac{\|\Pi_\disc w\|_{L^2(\O)}}{\|\nabla_\disc w\|_{L^2(\O)^{d}}}},
\end{equation}
%-----------
\begin{equation}\label{SD}
S_\disc(\varphi)=\min_{w\in X_{\disc,0}}\| \Pi_\disc w - \varphi \|_{L^2(\O)} 
+ \| \nabla_\disc w - \nabla \varphi \|_{L^2(\O)^{d}},\mbox{ and }
\end{equation}
%----------------------
\begin{equation}\label{WD}
W_\disc(\bphi)=
 \sup_{w\in X_{\disc,0}\setminus \{0\}}\frac{1}{\|\nabla_\disc w\|_{L^2(\O)^d}} \Big|\int_{\O}(\nabla_\disc w\cdot \bphi + \Pi_\disc w \cdot\mathrm{div} (\bphi)) \ud \x
\Big|.
\end{equation}
%-----------------------------------------
From the definition of $C_\disc$, we can derive the following discrete Poincar\'e inequality, for all $w\in X_{\disc,0}$,
\begin{equation}\label{ponc-enq}
\|\Pi_\disc w\|_{L^2(\O)} \leq C_\disc \|\nabla_\disc w\|_{L^2(\O)^d}.
\end{equation}
%------------------------------------------

Using an implicit Euler scheme in time and the space-time gradient discretisation illustrated above, we can define a generic discrete scheme of the problem \eqref{rm-weak} called a gradient scheme. In what follows, let $t^{(0)}=0<t^{(1)}<....<t^{(N)}=T$ $(N\in\NN^\star)$ be a discrete times and we set the time step $\delta t^{(n+\frac{1}{2})}:=t^{(n+1)}-t^{(n)}$, $n\in\{0,...,N-1\}$.
%------------------------
\begin{definition}[Gradient Scheme]
If $\disc$ is a space-time gradient discretisation defined in Definition \ref{def-gd}, the corressponding gradient scheme for the problem \eqref{rm-weak} is seeking $(u^{(n)},v^{(n)})_{n=0,...,N} \in X_{\disc,0}^{N+1} \times X_{\disc,0}^{N+1}$ such that $(u^{(0)}, v^{(0)})=(J_\disc u_{\rm ini}, J_\disc v_{\rm ini})$ and for all $n=0,...,N-1$, $u^{(n+1)}$ and $v^{(n+1)}$ satisfy the following equalities

\begin{subequations}\label{rm-disc-pblm}
\begin{equation}\label{rm-disc-pblm1}
\begin{aligned}
&\frac{1}{\delta t^{(n+\frac{1}{2})}}\dsp\int_\O (\Pi_\disc u^{(n+1)}(\x)-\Pi_\disc u^{(n)}(\x)) \Pi_\disc \varphi(\x)
+\dsp\int_\O \Lambda_1(\x)\nabla_\disc u^{(n+1)}(\x) \cdot \nabla_\disc \varphi(\x)\ud \x\\
&\qquad=\dsp\int_\O F_1(\Pi_\disc u^{(n+1)}(\x), \Pi_{\disc} v^{(n+1)}(\x))\Pi_\disc \varphi(\x) \ud \x, \quad \forall \varphi \in X_{\disc,0},
\end{aligned}
\end{equation}
\begin{equation}\label{rm-disc-pblm2}
\begin{aligned}
&\frac{1}{\delta t^{(n+\frac{1}{2})}}\dsp\int_\O (\Pi_\disc v^{(n+1)}(\x)-\Pi_\disc v^{(n)}(\x)) \Pi_\disc \psi(\x)
+\dsp\int_\O \Lambda_2(\x)\nabla_\disc v^{(n+1)}(\x) \cdot \nabla_\disc \psi(\x)\ud \x\\
&\qquad=\dsp\int_\O F_2(\Pi_\disc u^{(n+1)}(\x), \Pi_{\disc} v^{(n+1)}(\x))\Pi_\disc \psi(\x) \ud \x, \quad \forall \psi \in X_{\disc,0}.
\end{aligned}
\end{equation}
\end{subequations}
\end{definition}
%-------------------------------------------

Let us now prove the existence and uniqueness of the solutions to the approximate scheme. The proof presented here is a generalisation of the one in \cite{24}, which is limited to a single equation.  

\begin{lemma}
Let Assumptions \ref{assump-rm} hold, and $\disc$ be a space-time gradient discretisation. Then for all $\delta t^{(n+\frac{1}{2})}\leq \frac{2}{L^2C_\disc^2(\underline\lambda_1+\underline\lambda_2)}$, the approximate scheme \eqref{rm-disc-pblm} has a unique solution. 
\end{lemma}
%------------------------------------------------------

\begin{proof}
At any time step $n+1$, the scheme \eqref{rm-disc-pblm} leads to a square system of non-linear equations on unkowns $u^{(n+1)}$ and $v^{(n+1)}$ if $u^{(n)}$ and $v^{(n)}$ are assumed to be exist and unique. For a given $\bw=(w_1,w_2) \in X_{\disc,0}^2$, the pair $(u,v)\in X_{\disc,0}^2$ is a unique solution to the linear square system:
\begin{equation}\label{ex-u-1}
\begin{aligned}
&\frac{1}{\delta t^{(n+\frac{1}{2})}}\dsp\int_\O \Pi_\disc(u-u^{(n)})(\x)\Pi_\disc \varphi(\x)
+\dsp\int_\O \Lambda_1(\x)\nabla_\disc u(\x) \cdot \nabla_\disc \varphi(\x)\ud \x\\
&\qquad=\dsp\int_\O F_1(\Pi_\disc w_1, \Pi_{\disc} w_2)\Pi_\disc \varphi(\x) \ud \x, \quad \forall \varphi \in X_{\disc,0},\mbox{ and}
\end{aligned}
%------------------------
\end{equation}
\begin{equation}\label{ex-v-1}
\begin{aligned}
&\frac{1}{\delta t^{(n+\frac{1}{2})}}\dsp\int_\O \Pi_\disc(v-v^{(n)})(\x)\Pi_\disc \psi(\x)
+\dsp\int_\O \Lambda_2(\x)\nabla_\disc v(\x) \cdot \nabla_\disc \psi(\x)\ud \x\\
&\qquad=\dsp\int_\O F_2(\Pi_\disc w_1, \Pi_{\disc} w_2)\Pi_\disc \psi(\x) \ud \x, \quad \forall \psi \in X_{\disc,0},
\end{aligned}
\end{equation}
%----------------------------------
Since $\|\nabla_\disc \cdot\|_{L^2(\O)^d}$ is a norm on the space $X_{\disc,0}$, we consider the normed space $(\|(\cdot,\cdot)\|_{X_{\disc,0}^2},\;X_{\disc,0}^2)$, where the function $\|(\cdot,\cdot)\|_{X_{\disc,0}^2}:X_{\disc,0}^2 \to [0,\infty)$ is defined by
\begin{equation}\label{new-norm}
\|(w_1,w_2)\|_{X_{\disc,0}^2}:=\|\nabla_\disc w_1\|_{L^2(\O)^d}+\|\nabla_\disc w_2\|_{L^2(\O)^d}, \quad \mbox{for all } w_1,w_2 \in X_{\disc,0}.
\end{equation}
Let $G(\bw)=(u,v)$, such that $G:X_{\disc,0}^2\to X_{\disc,0}^2$ define a mapping. With the existence and uniqueness of $(u^{(0)},v^{(0)})$, to prove the existence and uniquness of the discrete solution $(u,v)$, it is suffices to show that $G$ is a contractive mapping in the space $(X_{\disc,0}^2\;,\|(\cdot,\cdot)\|_{X_{\disc,0}^2})$, thanks to Brouwer’s fixed point theorem.
%--------------------------------
%We remark that it defines a norm on the space $X_{\disc,0}^2$ since $\|\nabla_\disc \cdot\|_{L^2(\O)^d}$ is a norm

For $\bw=(w_1,w_2)$, $\tilde\bw=(\tilde w_1,\tilde w_2) \in X_{\disc,0}^2$ such that $(u,v)=G(\bw)$ and $(\tilde u,\tilde v)=G(\tilde\bw)$, it follows that 
%-------------------------
\begin{equation}\label{ex-u-1-a}
\begin{aligned}
&\frac{1}{\delta t^{(n+\frac{1}{2})}}\dsp\int_\O \Pi_\disc(u-u^{(n)})(\x)\Pi_\disc \varphi(\x)
+\dsp\int_\O \Lambda_1(\x)\nabla_\disc u(\x) \cdot \nabla_\disc \varphi(\x)\ud \x\\
&\qquad=\dsp\int_\O F_1(\Pi_\disc w_1, \Pi_{\disc} w_2)\Pi_\disc \varphi(\x) \ud \x, \quad \forall \varphi \in X_{\disc,0},
\end{aligned}
\end{equation}
%-------------------------------
\begin{equation}\label{ex-u-1-b}
\begin{aligned}
&\frac{1}{\delta t^{(n+\frac{1}{2})}}\dsp\int_\O \Pi_\disc(\tilde u-u^{(n)})(\x)\Pi_\disc \varphi(\x)
+\dsp\int_\O \Lambda_1(\x)\nabla_\disc \tilde u(\x) \cdot \nabla_\disc \varphi(\x)\ud \x\\
&\qquad=\dsp\int_\O F_1(\Pi_\disc \tilde w_1, \Pi_{\disc} \tilde w_2)\Pi_\disc \varphi(\x) \ud \x, \quad \forall \varphi \in X_{\disc,0},
\end{aligned}
\end{equation}
%----------------------------------
\begin{equation}\label{ex-v-1-a}
\begin{aligned}
&\frac{1}{\delta t^{(n+\frac{1}{2})}}\dsp\int_\O \Pi_\disc(v-v^{(n)})(\x)\Pi_\disc \psi(\x)
+\dsp\int_\O \Lambda_2(\x)\nabla_\disc v(\x) \cdot \nabla_\disc \psi(\x)\ud \x\\
&\qquad=\dsp\int_\O F_2(\Pi_\disc w_1, \Pi_{\disc} w_2)\Pi_\disc \psi(\x) \ud \x, \quad \forall \psi \in X_{\disc,0},\mbox{ and }
\end{aligned}
\end{equation}
%--------------------
\begin{equation}\label{ex-v-1-b}
\begin{aligned}
&\frac{1}{\delta t^{(n+\frac{1}{2})}}\dsp\int_\O \Pi_\disc(\tilde v-v^{(n)})(\x)\Pi_\disc \psi(\x)
+\dsp\int_\O \Lambda_2(\x)\nabla_\disc \tilde v(\x) \cdot \nabla_\disc \psi(\x)\ud \x\\
&\qquad=\dsp\int_\O F_2(\Pi_\disc \tilde w_1, \Pi_{\disc} \tilde w_2)\Pi_\disc \psi(\x) \ud \x, \quad \forall \psi \in X_{\disc,0}.
\end{aligned}
\end{equation}
%---------------------------
Subtract Equation \eqref{ex-u-1-b} from Equation \eqref{ex-u-1-a} (resp. Equation \eqref{ex-v-1-b} from \eqref{ex-v-1-a}) to obtain for all $\varphi,\; \psi \in X_{\disc,0}$, 
\begin{equation}\label{ex-u-2}
\begin{aligned}
&\frac{1}{\delta t^{(n+\frac{1}{2})}}\dsp\int_\O \Pi_\disc(u-\tilde u)(\x)\Pi_\disc \varphi(\x)
+\dsp\int_\O \Lambda_1(\x)\nabla_\disc (u-\tilde u)(\x) \cdot \nabla_\disc \varphi(\x)\ud \x\\
&\qquad=\dsp\int_\O \Big(F_1(\Pi_\disc w_1, \Pi_{\disc} w_2)
-F_1(\Pi_\disc \tilde w_1, \Pi_{\disc} \tilde w_2) \Big)\Pi_\disc \varphi(\x) \ud \x, \mbox{ and }
\end{aligned}
\end{equation}
%---------------------------
\begin{equation}\label{ex-v-2}
\begin{aligned}
&\frac{1}{\delta t^{(n+\frac{1}{2})}}\dsp\int_\O \Pi_\disc(v-\tilde v)(\x)\Pi_\disc \psi(\x)
+\dsp\int_\O \Lambda_2(\x)\nabla_\disc (v-\tilde v)(\x) \cdot \nabla_\disc \psi(\x)\ud \x\\
&\qquad=\dsp\int_\O \Big(F_2(\Pi_\disc w_1, \Pi_{\disc} w_2)
-F_2(\Pi_\disc \tilde w_1, \Pi_{\disc} \tilde w_2) \Big)\Pi_\disc \psi(\x) \ud \x.
\end{aligned}
\end{equation}
%-----------------------------------------------
Setting $\varphi=u-\tilde u$ in \eqref{ex-u-2} and $\psi=v-\tilde v$ in \eqref{ex-v-2}, we have
\begin{equation}\label{ex-u-3}
\begin{aligned}
&\frac{1}{\delta t^{(n+\frac{1}{2})}}\|\Pi_\disc(u-\tilde u)\|_{L^2(\O)}^2
+\frac{\underline\lambda_1}{\delta t^{(n+\frac{1}{2})}}\|\nabla_\disc(u-\tilde u)\|_{L^2(\O)^d}^2\\
&\qquad\leq\dsp\int_\O \Big(F_1(\Pi_\disc w_1, \Pi_{\disc} w_2)
-F_1(\Pi_\disc \tilde w_1, \Pi_{\disc} \tilde w_2) \Big)\Pi_\disc(u-\tilde u)(\x) \ud \x, \mbox{ and }
\end{aligned}
\end{equation}
%--------------------------------------
\begin{equation}\label{ex-v-3}
\begin{aligned}
&\frac{1}{\delta t^{(n+\frac{1}{2})}}\|\Pi_\disc(v-\tilde v)\|_{L^2(\O)}^2
+\frac{\underline\lambda_2}{\delta t^{(n+\frac{1}{2})}}\|\nabla_\disc(v-\tilde v)\|_{L^2(\O)^d}^2\\
&\qquad\leq\dsp\int_\O \Big(F_2(\Pi_\disc w_1, \Pi_{\disc} w_2)
-F_2(\Pi_\disc \tilde w_1, \Pi_{\disc} \tilde w_2) \Big)\Pi_\disc(v-\tilde v)(\x) \ud \x.
\end{aligned}
\end{equation}
%----------------------------------
Regarding the right-hand side of the above inequalities, applying the Cauchy–Schwarz inequality, the Lipschitz continuity assumptions for $F_1$ and $F_2$, and Young's inequality, one has
\begin{equation}\label{ex-u-4}
\begin{aligned}
&\dsp\int_\O \Big(F_1(\Pi_\disc w_1, \Pi_{\disc} w_2)
-F_1(\Pi_\disc \tilde w_1, \Pi_{\disc} \tilde w_2) \Big)\Pi_\disc(u-\tilde u)(\x) \ud \x\\
&\leq L\Big(\|\Pi_\disc w_1 - \Pi_\disc \tilde w_1\|_{L^2(\O)} + \|\Pi_\disc w_2 - \Pi_\disc \tilde w_2\|_{L^2(\O)^d}  \Big)\|\Pi_\disc(u-\tilde u)\|_{L^2(\O)}\\
&\leq \frac{L^2}{2}\Big(\|\Pi_\disc w_1 - \Pi_\disc \tilde w_1\|_{L^2(\O)} + \|\Pi_\disc w_2 - \Pi_\disc \tilde w_2\|_{L^2(\O)}  \Big)^2\\
&\quad\quad\quad+\frac{1}{2}\|\Pi_\disc(u-\tilde u)\|_{L^2(\O)}^2, \mbox{ and}
\end{aligned}
\end{equation}
%-----------------------------------------
\begin{equation}\label{ex-v-4}
\begin{aligned}
&\dsp\int_\O \Big(F_2(\Pi_\disc w_1, \Pi_{\disc} w_2)
-F_2(\Pi_\disc \tilde w_1, \Pi_{\disc} \tilde w_2) \Big)\Pi_\disc(v-\tilde v)(\x) \ud \x\\
&\leq L\Big(\|\Pi_\disc w_1 - \Pi_\disc \tilde w_1\|_{L^2(\O)} + \|\Pi_\disc w_2 - \Pi_\disc \tilde w_2\|_{L^2(\O)^d}  \Big)\|\Pi_\disc(v-\tilde v)\|_{L^2(\O)}\\
&\leq \frac{L^2}{2}\Big(\|\Pi_\disc w_1 - \Pi_\disc \tilde w_1\|_{L^2(\O)} + \|\Pi_\disc w_2 - \Pi_\disc \tilde w_2\|_{L^2(\O)^d}  \Big)^2\\
&\quad\quad\quad+\frac{1}{2}\|\Pi_\disc(v-\tilde v)\|_{L^2(\O)}^2.
\end{aligned}
\end{equation}
%-------------------
Substitute Inequality \eqref{ex-u-4} into \eqref{ex-u-3} (resp. Inequality \eqref{ex-v-4} into Equation \eqref{ex-u-3}) to get
\begin{equation*}
\begin{aligned}
\frac{\underline\lambda_1}{\delta t^{(n+\frac{1}{2})}}\|\nabla_\disc(u&-\tilde u)\|_{L^2(\O)^d}^2\\
&\quad\leq \frac{L^2}{2}\Big(\|\Pi_\disc w_1 - \Pi_\disc \tilde w_1\|_{L^2(\O)} + \|\Pi_\disc w_2 - \Pi_\disc \tilde w_2\|_{L^2(\O)}  \Big)^2, \mbox{ and }
\end{aligned}
\end{equation*}
%----------------------------------
\begin{equation*}
\begin{aligned}
\frac{\underline\lambda_2}{\delta t^{(n+\frac{1}{2})}}\|\nabla_\disc(v&-\tilde v)\|_{L^2(\O)^d}^2\\
&\quad\leq \frac{L^2}{2}\Big(\|\Pi_\disc w_1 - \Pi_\disc \tilde w_1\|_{L^2(\O)} + \|\Pi_\disc w_2 - \Pi_\disc \tilde w_2\|_{L^2(\O)}  \Big)^2.
\end{aligned}
\end{equation*}
%--------------------------------------------
Taking the square root of both sides yields
\begin{equation*}\label{ex-u-5}
\begin{aligned}
\sqrt{ \frac{\underline\lambda_1}{\delta t^{(n+\frac{1}{2})}} }\|\nabla_\disc(u&-\tilde u)\|_{L^2(\O)^d}\\
&\quad\leq \frac{\sqrt{2}L}{2}\Big(\|\Pi_\disc w_1 - \Pi_\disc \tilde w_1\|_{L^2(\O)} + \|\Pi_\disc w_2 - \Pi_\disc \tilde w_2\|_{L^2(\O)}  \Big), \mbox{ and }
\end{aligned}
\end{equation*}
%----------------------------------
\begin{equation*}\label{ex-v-5}
\begin{aligned}
\sqrt{ \frac{\underline\lambda_2}{\delta t^{(n+\frac{1}{2})}} }\|\nabla_\disc(v&-\tilde v)\|_{L^2(\O)^d}\\
&\quad\leq \frac{\sqrt{2}L}{2}\Big(\|\Pi_\disc w_1 - \Pi_\disc \tilde w_1\|_{L^2(\O)} + \|\Pi_\disc w_2 - \Pi_\disc \tilde w_2\|_{L^2(\O)}  \Big).
\end{aligned}
\end{equation*}
%---------------------------
Gathering the above inequalities leads to, due to the discrete Poincar\'e inequality \eqref{ponc-enq} and the relation $\sqrt{\underline\lambda_1+\underline\lambda_2}\leq \sqrt{\underline\lambda_1}+\sqrt{\underline\lambda_2}$,
\begin{equation*}\label{ex-6}
\begin{aligned}
&\sqrt{ \frac{\underline\lambda_1+\underline\lambda_2}{\delta t^{(n+\frac{1}{2})}} }\Big(\|\nabla_\disc(u-\tilde u)\|_{L^2(\O)^d}
+\|\nabla_\disc(v-\tilde v)\|_{L^2(\O)^d} \Big)\\
&\leq \sqrt{2}C_\disc L \Big(\|\nabla_\disc w_1 - \nabla_\disc \tilde w_1\|_{L^2(\O)^d} + \|\nabla_\disc w_2 - \nabla_\disc \tilde w_2\|_{L^2(\O)^d}  \Big),
\end{aligned}
\end{equation*}
%---------------------------
%-----------------------------------
which provides, thanks to the definition of the norm given by \eqref{new-norm}
\begin{equation*}\label{ex-8}
\begin{aligned}
\sqrt{ \frac{\underline\lambda_1+\underline\lambda_2}{\delta t^{(n+\frac{1}{2})}} }\|(u-\tilde u,v-\tilde v)\|_{X_{\disc,0}^2}
\leq \sqrt{2}C_\disc L\|(w_1 -\tilde w_1,w_2 -\tilde w_2)\|_{X_{\disc,0}^2}.
\end{aligned}
\end{equation*}
%-----------------------
One is able to write
\begin{equation*}\label{ex-9}
\begin{aligned}
\sqrt{ \frac{\underline\lambda_1+\underline\lambda_2}{\delta t^{(n+\frac{1}{2})}} }\|(u,v)-(\tilde u,\tilde v)\|_{X_{\disc,0}^2}
\leq \sqrt{2}C_\disc L\|(w_1,w_2)-(\tilde w_1,\tilde w_2)\|_{X_{\disc,0}^2}.
\end{aligned}
\end{equation*}
%--------------------------------
Using the relation $(u,v)=G(\bw)$ and $(\tilde u,\tilde v)=G(\tilde \bw)$, we obtain
\begin{equation*}\label{ex-10}
\begin{aligned}
\|G(\bw)-G(\tilde \bw)\|_{X_{\disc,0}^2}
\leq\sqrt{ \frac{2\delta t^{(n+\frac{1}{2})}} {\underline\lambda_1+\underline\lambda_2} }C_\disc L\|\bw-\tilde \bw\|_{X_{\disc,0}^2},
\end{aligned}
\end{equation*}
which shows that $G$ is a contractive mapping.
\end{proof}

%%%%%%%%%%%%%%%%%%%%%%%%%%%%%%%%%%%%%%%%%%%%%%%%%%%%%%%%%%%%%%%%%%%
%SECTIION
\section{Convergence rates}\label{sec-3}
Here, we state the main results of our contribution along with error estimates. The established convergence rates depend on the space size and the time discretisation. The convergence rates relative to space size depend on the three indicators $C_\disc$, $S_\disc$ and $W_\disc$. Let $h_\disc$ denote the space size defined in \cite[Definition 2.22]{21}. With respect to the continuously embedded spaces $W^{2,\infty}(\O)$ and $W^{1,\infty}(\O)^d$, it is defined by
\[
h_\disc:=\dsp\max\Big( 
\dsp\sup_{\varphi\in W^{2,\infty}(\O)\setminus\{0\}}\frac{S_\disc(\varphi)}{\|\varphi\|_{W^{2,\infty}(\O)}},  
\dsp\sup_{\varphi\in W^{1,\infty}(\O)^d\setminus\{0\}}\frac{W_\disc(\bphi)}{\|\bphi\|_{W^{1,\infty}(\O)^d}}
\Big),
\]
%----------------------------------------
and it therefore satisfies
\begin{subequations}\label{hd}
\begin{equation}\label{hd-1}
\forall \varphi \in W^{2,\infty}(\O) \cap H_0^1(\O),\quad S_\disc(\varphi)\leq h_\disc \|\varphi\|_{W^{2,\infty}(\O)},
\end{equation}
\begin{equation}\label{hd-2}
\forall \bphi \in W^{1,\infty}(\O)^d \cap H_0^1(\O)^d,\quad W_\disc(\bphi)\leq h_\disc\|\bphi\|_{W^{1,\infty}(\O)^d}.
\end{equation}
\end{subequations}
%----------------------------------------------

It is worth mentioning that \cite[Lemma 1]{22} provides us with a linear spatial interpolator $P_\disc:H_0^1(\O)\to X_{\disc,0}$ defined by
\begin{equation}\label{PD}
P_\disc(w):=\arg\min_{w\in X_{\disc,0}} \|\Pi_\disc w - \varphi\|_{L^2(\O)}
+\|\nabla_\disc w - \nabla\varphi\|_{L^2(\O)},
\end{equation}
and it meets the following relation
\begin{equation}\label{PD-r}
S_\disc(\varphi)= \Big( \|\Pi_\disc P_\disc\varphi - \varphi\|_{L^2(\O)}
+\|\nabla_\disc P_\disc\varphi - \nabla\varphi\|_{L^2(\O)} \Big)^\frac{1}{2}, \quad \forall \varphi\in H_0^1(\O).
\end{equation}
%---------------------------------------------
Let $\varphi\in H_0^1(\O)$. For any arbitrary $w\in X_{\disc,0}$, the definition of $S_\disc$ and the interpolant $P_\disc$ show that
\begin{equation}\label{new-eq-proof-100}
\begin{aligned}
\Big( \|\Pi_\disc P_\disc\varphi &- \varphi\|_{L^2(\O)}^2
+\|\nabla_\disc P_\disc\varphi - \nabla\varphi\|_{L^2(\O)^d}^2  \Big)^{1/2}
\\
&\leq \Big( \|\Pi_\disc w - \varphi\|_{L^2(\O)}^2
+\|\nabla_\disc w - \nabla\varphi\|_{L^2(\O)^d}^2  \Big)^{1/2}\\
&\leq \sqrt{2}S_\disc(\varphi).
\end{aligned}
\end{equation}
%-----------------------------------------------------

%%%%%% CONVERGENCE THEOREM %%%%%%%%%%%%%
\begin{theorem}\label{thm-err-rm} 
Under Assumptions \ref{assump-rm}, let $\disc$ be a space-time gradient discretisation, $(\bar u,\bar v)$ and $(u,v)$ be the solutions to the continous problem \eqref{rm-weak} and to the approximate problem \eqref{rm-disc-pblm}, respectively, and $\delta t=\max_{n\in\{0,...,N-1\}}\delta t^{(n+\frac{1}{2})}$. If $(\bar u,\bar v) \in (W^{1,\infty}(0,T;W^{2,\infty}(\O)))^2$ and $\Lambda_1$ and $\Lambda_2$ are Lipschitz continuous, then the following error estimates hold:
%---------------------------
\begin{subequations}\label{eq-error}
\begin{equation}\label{eq-error-1}
\begin{aligned}
\max_{t\in[0,T]}\Big\| \Pi_\disc u(\cdot,t)-\bar u(\cdot,t) \Big\|_{L^2(\O)}
&+\max_{t\in[0,T]}\Big\| \Pi_\disc v(\cdot,t)-\bar v(\cdot,t) \Big\|_{L^2(\O)}\\
&\leq C(\delta t+h_\disc+E_\disc^0+\widetilde E_\disc^0),
\end{aligned}
\end{equation}
\begin{equation}\label{eq-error-2}
\begin{aligned}
\Big\| \nabla_\disc u-\nabla\bar u \Big\|_{L^2(\O\times(0,T))^d}
&+\Big\| \nabla_\disc v-\nabla\bar v \Big\|_{L^2(\O\times(0,T))^d}
\\
&\leq C(\delta t+h_\disc+E_\disc^0+\widetilde E_\disc^0).
\end{aligned}
\end{equation}
\end{subequations}

%-------------------------------------
Here $C>0$ does not depend on discrete data, and $E_\disc^0$ and $\widetilde E_\disc^0$ stand to the interpolation of the initial conditions defined by
\begin{equation}
E_\disc^0=\|u_{\rm ini}-\Pi_\disc J_\disc u_{\rm ini} \|_{L^2(\O)}
\mbox{ and }
\widetilde E_\disc^0=\|v_{\rm ini}-\Pi_\disc J_\disc v_{\rm ini} \|_{L^2(\O)}.
\end{equation}
\end{theorem}
%-------------------------------

\begin{proof}
The proof is inspired from \cite{22}. Here, the constants $C_i>0$ do not depend on the parameters of the discretisations. Also, for $0\leq n \leq N-1$, we use the notation
\[
f^{(n+1)}(\x):=\frac{1}{\delta t^{(n+\frac{1}{2})}}\dsp\int_{t^{(n)}}^{t^{(n+1)}} f(\x,t) \ud t, \mbox{ where $f=\bar u, \bar v, \dr_t \bar u, \dr_t \bar v, F_1 $ or $ F_2$.} 
\]  
Based on the assumptions stated in the theorem, we observe that $\nabla\bar u,\nabla\bar v: [0,T] \to L^2(\O)^d$ are Lipschitz--continuous. Consequently, employing \eqref{PD} to $\bar u(t^{(n+1)})$ and to $\bar v(t^{(n+1)})$, we infer thanks to \eqref{hd-1}
\begin{equation}\label{eq-proof-1-u}
\begin{aligned}
\Big\| \nabla\bar u^{(n+1)}&-\nabla_\disc  P_\disc \bar u(t^{(n+1)}) \Big\|_{L^2(\O)^d}\\
&\leq \Big\| \nabla\bar u^{(n+1)}-\nabla\bar u(t^{(n+1)}) \Big\|_{L^2(\O)^d} + S_\disc(\bar u(t^{(n+1)}))\\
&\leq C_1(\delta t+h_\disc),\mbox{ and }
\end{aligned}
\end{equation}
%-------------------------------------
\begin{equation}\label{eq-proof-1-v}
\begin{aligned}
\Big\| \nabla\bar v^{(n+1)}&-\nabla_\disc  P_\disc \bar v(t^{(n+1)}) \Big\|_{L^2(\O)^d}\\
&\leq \Big\| \nabla\bar v^{(n+1)}-\nabla\bar v(t^{(n+1)}) \Big\|_{L^2(\O)^d} + S_\disc(\bar v(t^{(n+1)}))\\
&\leq C_2(\delta t+h_\disc).
\end{aligned}
\end{equation}
%-----------------------------------------------------------
We note that both $\| \dr_t \bar u^{(n+1)} \|_{H^2(\O)}$ and $\| \dr_t \bar v^{(n+1)} \|_{H^2(\O)}$ remain bounded independently of $n$. Direct application of the interpolant $P_\disc$ defined by \eqref{PD} to $\varphi:=\dr_t\bar u^{(n+1)}=\frac{\bar u(t^{(n+1)})-\bar u(t^{(n)})}{\delta t^{(n+\frac{1}{2})}}$, and to $\varphi:=\dr_t\bar v^{(n+1)}=\frac{\bar v(t^{(n+1)})-\bar v(t^{(n)})}{\delta t^{(n+\frac{1}{2})}}$ yields, thanks to the linearity of the interpolant $P_\disc$ and the relation \eqref{hd-1}
\begin{equation}\label{eq-proof-2-u}
\begin{aligned}
\Big\| \frac{\Pi_\disc  P_\disc\bar u(t^{(n+1)})-\Pi_\disc  P_\disc \bar u(t^{(n)})}
{\delta t^{(n+\frac{1}{2})}}
-\partial_t \bar u^{(n+1)} \Big\|_{L^2(\O)}
&\leq S_\disc( \partial_t \bar u^{(n+1)})\\
&\leq C_3h_\disc, \mbox{ and }
\end{aligned}
\end{equation}
%---------------------------
\begin{equation}\label{eq-proof-2-v}
\begin{aligned}
\Big\| \frac{\Pi_\disc  P_\disc\bar v(t^{(n+1)})-\Pi_\disc  P_\disc \bar v(t^{(n)})}
{\delta t^{(n+\frac{1}{2})}}
-\partial_t \bar v^{(n+1)} \Big\|_{L^2(\O)}
&\leq S_\disc( \partial_t \bar v^{(n+1)})\\
&\leq C_4h_\disc.
\end{aligned}
\end{equation}
%---------------------------------
Given that $\Lambda_1\nabla\bar u^{(n+1)}$ and $\Lambda_2\nabla\bar v^{(n+1)}$ belong to $H_{\div}(\O)$, we can apply the limit-conformity indicotor \eqref{WD} to $\bpsi=:\Lambda_1\nabla\bar u^{(n+1)}$ and to $\bpsi=:\Lambda_2\nabla\bar v^{(n+1)}$ to obtain, thanks to \eqref{hd-2}
\begin{equation}\label{eq-limi-conf-u}
\begin{aligned}
\dsp\int_\O \Big(\Pi_\disc w(\x)&\div( \Lambda_1\nabla\bar u^{(n+1)}(\x))
+\Lambda_1\nabla\bar u^{(n+1)}(\x)\cdot \nabla_\disc w(\x) \Big) \ud \x \\
&\leq W_\disc(\Lambda_1\nabla\bar u^{(n+1)}) \| \nabla_\disc w \|_{L^2(\O)^d}\\
&\leq C_5h_\disc \| \nabla_\disc w \|_{L^2(\O)^d},\quad \forall w\in X_{\disc,0} \mbox{ and }
\end{aligned}
\end{equation}
%-------------------------------
\begin{equation}\label{eq-limi-conf-v}
\begin{aligned}
\int_\O \Big(\Pi_\disc w(\x)&\div( \Lambda_2\nabla\bar v^{(n+1)}(\x))
+\Lambda_2\nabla\bar v^{(n+1)}(\x)\cdot \nabla_\disc w(\x) \Big) \ud \x \\
&\leq W_\disc(\Lambda_2\nabla\bar v^{(n+1)}) \| \nabla_\disc w \|_{L^2(\O)^d}\\
&\leq C_6h_\disc \| \nabla_\disc w \|_{L^2(\O)^d},\quad \forall w\in X_{\disc,0}.
\end{aligned}
\end{equation}
%--------------------------------
Notice that the solution $(\bar u,\bar v)$ is sufficiently regular to guarantee that Equations \eqref{rm-strong1}--\eqref{rm-strong2} hold a.e. in space and time. Averaging over time in $(t^{(n)},t^{(n+1)})$ gives $\dr_t \bar u^{(n+1)}-F_1(\bar u^{(n+1)},\bar v^{(n+1)}) = \div(\Lambda_1\nabla\bar u^{(n+1)})$ and $\dr_t \bar v^{(n+1)}-F_2(\bar u^{(n+1)},\bar v^{(n+1)}) = \div(\Lambda_2\nabla\bar v^{(n+1)})$. Therefore, inserting these equalities in \eqref{eq-limi-conf-u} and \eqref{eq-limi-conf-v} leads to 
\begin{equation}\label{eq-proof-3-u}
\begin{aligned}
\int_\O \Pi_\disc w(\x) \Big( \dr_t\bar u^{(n+1)}(\x) &- F_1(\bar u^{(n+1)},\bar v^{(n+1)}) \Big) \ud \x\\
&\quad+\dsp\int_\O\Lambda_1\nabla\bar u^{(n+1)}(\x)\cdot \nabla_\disc w(\x) \ud \x \\
&\leq C_5h_\disc \| \nabla_\disc w \|_{L^2(\O)^d},\quad \forall w\in X_{\disc,0} \mbox{ and }
\end{aligned}
\end{equation}
%-------------------------------
\begin{equation}\label{eq-proof-3-v}
\begin{aligned}
\int_\O \Pi_\disc w(\x) \Big( \dr_t\bar v^{(n+1)}(\x) &- F_2(\bar u^{(n+1)},\bar v^{(n+1)}) \Big)\\
&\quad+\dsp\int_\O\Lambda_2\nabla\bar v^{(n+1)}(\x)\cdot \nabla_\disc w(\x)  \ud \x \\
&\leq C_6h_\disc \| \nabla_\disc w \|_{L^2(\O)^d}\quad \forall w\in X_{\disc,0}.
\end{aligned}
\end{equation}
%--------------------------------------
By introducing $F_1(\Pi_\disc u^{(n+1)},\Pi_\disc v^{(n+1)})$ and $F_2(\Pi_\disc u^{(n+1)},\Pi_\disc v^{(n+1)})$, we infer
\begin{equation}\label{eq-proof-4-u}
\begin{aligned}
\int_\O &\Pi_\disc w(\x) \Big( \dr_t\bar u^{(n+1)}(\x)- F_1(\Pi_\disc u^{(n+1)},\Pi_\disc v^{(n+1)}) \Big) \ud \x\\
&\quad+\int_\O \Pi_\disc w(\x)\Big( F_1(\Pi_\disc u^{(n+1)},\Pi_\disc v^{(n+1)}) - F_1(\bar u^{(n+1)},\bar v^{(n+1)}) \Big) \ud \x\\
&\quad+\int_\O\Lambda_1\nabla\bar u^{(n+1)}(\x)\cdot \nabla_\disc w(\x)  \ud \x \\
&\leq C_5h_\disc \| \nabla_\disc w \|_{L^2(\O)^d},\quad \forall w\in X_{\disc,0}, \mbox{ and }
\end{aligned}
\end{equation}
%----------------------------------
\begin{equation}\label{eq-proof-4-v}
\begin{aligned}
\int_\O &\Pi_\disc w(\x) \Big( \dr_t\bar v^{(n+1)}(\x)- F_2(\Pi_\disc u^{(n+1)},\Pi_\disc v^{(n+1)}) \Big) \ud \x\\
&\quad+\int_\O \Pi_\disc w(\x)\Big( F_2(\Pi_\disc u^{(n+1)},\Pi_\disc v^{(n+1)}) - F_2(\bar u^{(n+1)},\bar v^{(n+1)}) \Big) \ud \x\\
&\quad+\int_\O\Lambda_2\nabla\bar v^{(n+1)}(\x)\cdot \nabla_\disc w(\x)  \ud \x \\
&\leq C_6h_\disc \| \nabla_\disc w \|_{L^2(\O)^d},\quad \forall w\in X_{\disc,0}.
\end{aligned}
\end{equation}
%----------------------------------
Concerning the first term in the above inequalities, from the gradient scheme \eqref{rm-disc-pblm}, and recalling the property that $\nabla\bar u^{(n+1)}$ and $\nabla\bar v^{(n+1)}$ are bounded in the space $W^{1,\infty}(\O)^d$, we infer
\begin{equation}\label{eq-proof-5-u}
\begin{aligned}
\int_\O &\Pi_\disc w(\x) \Big( \dr_t\bar u^{(n+1)}(\x)- \delta_\disc^{(n+\frac{1}{2})}u(\x) \Big) \ud \x\\
&\quad+\int_\O \Lambda_1\Big(\nabla\bar u^{(n+1)}(\x)-\nabla_\disc u^{(n+1)} \Big)\cdot \nabla_\disc w(\x)  \ud \x \\
&\leq\int_\O \Pi_\disc w(\x)\Big(F_1(\bar u^{(n+1)},\bar v^{(n+1)})-F_1(\Pi_\disc u^{(n+1)},\Pi_\disc v^{(n+1)}) \Big) \ud \x\\
&\quad+C_7h_\disc \| \nabla_\disc w \|_{L^2(\O)^d},\quad \forall w\in X_{\disc,0}, \mbox{ and }
\end{aligned}
\end{equation}
%------------------------------
\begin{equation}\label{eq-proof-5-v}
\begin{aligned}
\int_\O &\Pi_\disc w(\x) \Big( \dr_t\bar v^{(n+1)}(\x)- \delta_\disc^{(n+\frac{1}{2})}v(\x) \Big) \ud \x\\
&\quad+\int_\O \Lambda_2\Big(\nabla\bar v^{(n+1)}(\x)-\nabla_\disc v^{(n+1)} \Big)\cdot \nabla_\disc w(\x)  \ud \x \\
&\leq\int_\O \Pi_\disc w(\x)\Big(F_2(\bar u^{(n+1)},\bar v^{(n+1)})-F_1(\Pi_\disc u^{(n+1)},\Pi_\disc v^{(n+1)}) \Big) \ud \x\\
&\quad+C_8h_\disc \| \nabla_\disc w \|_{L^2(\O)^d}\quad \forall w\in X_{\disc,0}.
\end{aligned}
\end{equation}
%------------------------------
Set the notions $E^{(k)}:=P_\disc\bar u(t^{(k)})-u^{(k)}$ and $\widetilde E^{(k)}:=P_\disc\bar v(t^{(k)})-v^{(k)}$, for $k=1,...,N$. We have
\[
\begin{aligned}
\delta_\disc^{(n+\frac{1}{2})}E:&=\frac{ \Pi_\disc E^{(n+1)}-\Pi_\disc E^{(n)} }{ \delta t^{(n+\frac{1}{2})} }\\
&=\Big(  \dsp\frac{ \Pi_\disc  P_\disc\bar u(t^{(n+1)})-\Pi_\disc  P_\disc(\bar u(t^{(n)})) }{ \delta t^{(n+\frac{1}{2})} }-\dr_t\bar u^{(n+1)} \Big)+\Big( \dr_t\bar u^{(n+1)}-\delta_\disc^{(n+\frac{1}{2})} u \Big),
\end{aligned}
\]
%----------------------------
\[
\begin{aligned}
\delta_\disc^{(n+\frac{1}{2})}\widetilde E:&=\frac{ \Pi_\disc \widetilde E^{(n+1)}-\Pi_\disc \widetilde E^{(n)} }{ \delta t^{(n+\frac{1}{2})} }\\
&=\Big(  \dsp\frac{ \Pi_\disc  P_\disc\bar v(t^{(n+1)})-\Pi_\disc  P_\disc(\bar v(t^{(n)})) }{ \delta t^{(n+\frac{1}{2})} }-\dr_t\bar v^{(n+1)} \Big)+\Big( \dr_t\bar v^{(n+1)}-\delta_\disc^{(n+\frac{1}{2})} v \Big),
\end{aligned}
\]
%---------------------
\[
\Lambda_1\nabla_\disc E^{(n+1)}= \Lambda_1\left( \nabla_\disc  (P_\disc\bar u(t^{(n+1)}))-\nabla\bar u^{(n+1)}  \right) + \Lambda_1\left(\nabla\bar u^{(n+1)}-\nabla_\disc u^{(n+1)}  \right),
\]
%---------------------------
\[
\Lambda_2\nabla_\disc \widetilde E^{(n+1)}= \Lambda_2\left( \nabla_\disc  (P_\disc\bar v(t^{(n+1)}))-\Lambda_2\nabla\bar v^{(n+1)}  \right) + \Lambda_2\left(\nabla\bar v^{(n+1)}-\nabla_\disc v^{(n+1)}  \right).
\]
%-------------------------------------
Using \eqref{eq-proof-5-u}, \eqref{eq-proof-2-u} and \eqref{eq-proof-1-u}, we derive
\begin{equation}\label{eq-proof-6-u}
\begin{aligned}
\int_\O &\Pi_\disc w(\x)\delta_\disc^{(n+\frac{1}{2})}E(\x) \ud \x
+\int_\O \nabla_\disc w(\x)\cdot \Lambda_1\nabla_\disc E^{(n+1)}(\x) \ud \x\\
&\quad\leq\int_\O \Pi_\disc w(\x)\Big(F_1(\bar u^{(n+1)},\bar v^{(n+1)})-F_1(\Pi_\disc u^{(n+1)},\Pi_\disc v^{(n+1)}) \Big) \ud \x\\
&\quad+C_9(\delta t+h_\disc)\|\nabla_\disc w\|_{L^2(\O)^d}\quad \forall w\in X_{\disc,0}.
\end{aligned}
\end{equation}
%--------------------------------
Also, using \eqref{eq-proof-5-v}, \eqref{eq-proof-2-v} and \eqref{eq-proof-1-v}, we derive
\begin{equation}\label{eq-proof-6-v}
\begin{aligned}
\int_\O &\Pi_\disc w(\x)\delta_\disc^{(n+\frac{1}{2})}\widetilde E(\x) \ud \x
+\int_\O \nabla_\disc w(\x)\cdot \Lambda_2\nabla_\disc \widetilde E^{(n+1)}(\x) \ud \x\\
&\quad\leq\int_\O \Pi_\disc w(\x)\Big(F_2(\bar u^{(n+1)},\bar v^{(n+1)})-F_2(\Pi_\disc u^{(n+1)},\Pi_\disc v^{(n+1)}) \Big) \ud \x\\
&\quad+C_{10}(\delta t+h_\disc)\|\nabla_\disc w\|_{L^2(\O)^d}\quad \forall w\in X_{\disc,0}.
\end{aligned}
\end{equation}
%------------------------------------
Set $w=\delta t^{(n+\frac{1}{2})}E^{(n+1)}$ in \eqref{eq-proof-6-u} and $w=\delta t^{(n+\frac{1}{2})}\widetilde E^{(n+1)}$ in \eqref{eq-proof-6-v}. Summing over $n = 0,...,m-1$ for some $m\in \{1,...,N\}$ yeilds
\begin{equation}\label{eq-proof-7-u}
\begin{aligned}
&\dsp\sum_{n=0}^{m-1} \int_\O \Pi_\disc E^{(n+1)}(\x)\Big(\Pi_\disc E^{(n+1)}(\x)-\Pi_\disc E^{(n)}(\x)\Big) \ud \x\\
&\quad+\underline\lambda_1\dsp\sum_{n=0}^{m-1} \delta t^{(n+\frac{1}{2})} \left\| \nabla_\disc E^{(n+1)} \right\|_{L^2(\O)^d}^2\\
&\leq\dsp\sum_{n=0}^{m-1} \delta t^{(n+\frac{1}{2})}\int_\O \Pi_\disc E^{(n+1)}(\x)\Big(F_1(\bar u^{(n+1)},\bar v^{(n+1)})-F_1(\Pi_\disc u^{(n+1)},\Pi_\disc v^{(n+1)}) \Big) \ud \x\\
&\quad+\dsp\sum_{n=0}^{m-1} \delta t^{(n+\frac{1}{2})}C_9(\delta t+h_\disc)\left\| \nabla_\disc E^{(n+1)} \right\|_{L^2(\O)^d},\mbox{ and }
\end{aligned}
\end{equation}
%-----------------------------------------
\begin{equation}\label{eq-proof-7-v}
\begin{aligned}
&\dsp\sum_{n=0}^{m-1} \int_\O \Pi_\disc \widetilde E^{(n+1)}(\x)\Big(\Pi_\disc \widetilde E^{(n+1)}(\x)-\Pi_\disc \widetilde E^{(n)}(\x)\Big) \ud \x\\
&+\underline\lambda_2\dsp\sum_{n=0}^{m-1} \delta t^{(n+\frac{1}{2})} \left\| \nabla_\disc \widetilde E^{(n+1)} \right\|_{L^2(\O)^d}^2\\
&\leq\dsp\sum_{n=0}^{m-1} \delta t^{(n+\frac{1}{2})}\int_\O \Pi_\disc \widetilde E^{(n+1)}(\x)\Big(F_2(\bar u^{(n+1)},\bar v^{(n+1)})-F_1(\Pi_\disc u^{(n+1)},\Pi_\disc v^{(n+1)}) \Big) \ud \x\\
&\quad+\dsp\sum_{n=0}^{m-1} \delta t^{(n+\frac{1}{2})}C_{10}(\delta t+h_\disc)\left\| \nabla_\disc \widetilde E^{(n+1)} \right\|_{L^2(\O)^d}.
\end{aligned}
\end{equation}
%------------------------------
Employing the relation $\beta(\beta-\alpha)=\frac{1}{2}\beta^2-\frac{1}{2}\alpha^2+\frac{1}{2}(\beta-\alpha)^2 \geq \frac{1}{2}\beta^2-\frac{1}{2}\alpha^2$ to the first term in the left-hand sides, and applying the Young's inequality with small positive parameters $ab\leq \frac{\varepsilon_1}{2} a^2 + \frac{1}{2\varepsilon_1} b^2$ to the last term in the right-hand sides, we obtain
\begin{equation}\label{eq-proof-8-u}
\begin{aligned}
&\frac{1}{2}\dsp\int_\O(\Pi_\disc E^{(m)}(\x))^2 \ud \x
+\underline\lambda_1\dsp\sum_{n=0}^{m-1} \delta t^{(n+\frac{1}{2})} \left\| \nabla_\disc E^{(n+1)} \right\|_{L^2(\O)^d}^2 \\
&\leq \frac{1}{2}\dsp\int_\O (\Pi_\disc E^{(0)}(\x))^2 \ud \x
+\frac{\varepsilon_1}{2}\dsp\sum_{n=0}^{m-1} \delta t^{(n+\frac{1}{2})} \left\| \nabla_\disc E^{(n+1)} \right\|_{L^2(\O)^d}^2\\
&\quad+\dsp\sum_{n=0}^{m-1} \delta t^{(n+\frac{1}{2})}\int_\O \Pi_\disc E^{(n+1)}(\x)\Big(F_1(\bar u^{(n+1)},\bar v^{(n+1)})-F_1(\Pi_\disc u^{(n+1)},\Pi_\disc v^{(n+1)}) \Big) \ud \x\\
&\quad+\frac{1}{2\varepsilon_1}\dsp\sum_{n=0}^{m-1} \delta t^{(n+\frac{1}{2})} C_9(\delta t+h_\disc)^2,\mbox{ and }
\end{aligned}
\end{equation}
%------------------------------
\begin{equation}\label{eq-proof-8-v}
\begin{aligned}
&\frac{1}{2}\dsp\int_\O(\Pi_\disc \widetilde E^{(m)}(\x))^2 \ud \x
+\underline\lambda_2\dsp\sum_{n=0}^{m-1} \delta t^{(n+\frac{1}{2})} \left\| \nabla_\disc \widetilde E^{(n+1)} \right\|_{L^2(\O)^d}^2 \\
&\leq \frac{1}{2}\dsp\int_\O (\Pi_\disc \widetilde E^{(0)}(\x))^2 \ud \x
+\frac{\varepsilon_2}{2}\dsp\sum_{n=0}^{m-1} \delta t^{(n+\frac{1}{2})} \left\| \nabla_\disc \widetilde E^{(n+1)} \right\|_{L^2(\O)^d}^2\\
&\quad+\dsp\sum_{n=0}^{m-1} \delta t^{(n+\frac{1}{2})}\int_\O \Pi_\disc \widetilde E^{(n+1)}(\x)\Big(F_2(\bar u^{(n+1)},\bar v^{(n+1)})-F_1(\Pi_\disc u^{(n+1)},\Pi_\disc v^{(n+1)}) \Big) \ud \x\\
&\quad+\frac{1}{2\varepsilon_2}\dsp\sum_{n=0}^{m-1} \delta t^{(n+\frac{1}{2})} C_{10}(\delta t+h_\disc)^2.
\end{aligned}
\end{equation}
%------------------------------------
We can evaluate the initial conditions as follows
\begin{equation}\label{eq-proof-9-u}
\begin{aligned}
\| \Pi_\disc E^{(0)} \|_{L^2(\O)} &\leq \| \Pi_\disc P_\disc \bar u(0)-\bar u(0) \|_{L^2(\O)}
+\| \bar u(0)-\Pi_\disc J_\disc\bar u(0) \|_{L^2(\O)}\\
&\leq C_{11}h_\disc + E_\disc^0,\mbox{ and }
\end{aligned}
\end{equation}
%---------------------------
\begin{equation}\label{eq-proof-9-v}
\begin{aligned}
\| \Pi_\disc \widetilde E^{(0)} \|_{L^2(\O)} &\leq \| \Pi_\disc P_\disc \bar v(0)-\bar v(0) \|_{L^2(\O)}
+\| \bar v(0)-\Pi_\disc J_\disc\bar v(0) \|_{L^2(\O)}\\
&\leq C_{12}h_\disc + \widetilde E_\disc^0.
\end{aligned}
\end{equation}
%------------------------------
Now, we will estimate the third term on the right side of Inequality \eqref{eq-proof-8-u}. By applying the Cauchy–Schwarz inequality and the Lipschitz continuity assumption for $F_1$, we deduce
\begin{equation}\label{eq-proof-F}
\begin{aligned}
\dsp\sum_{n=0}^{m-1} &\delta t^{(n+\frac{1}{2})}\int_\O \Pi_\disc E^{(n+1)}(\x)\Big(F_1(\bar u^{(n+1)},\bar v^{(n+1)})-F_1(\Pi_\disc u^{(n+1)},\Pi_\disc v^{(n+1)}) \Big) \ud \x\\
&\leq \dsp\sum_{n=0}^{m-1} \delta t^{(n+\frac{1}{2})} L\|\Pi_\disc E^{(n+1)}\|_{L^2(\O)}
\|\bar u^{(n+1)}-\Pi_\disc u^{(n+1)}\|_{L^2(\O)}:=\mathbb T_1\\
&\quad+\dsp\sum_{n=0}^{m-1} \delta t^{(n+\frac{1}{2})} L\|\Pi_\disc E^{(n+1)}\|_{L^2(\O)}
\|\bar v^{(n+1)}-\Pi_\disc v^{(n+1)}\|_{L^2(\O)}:=\mathbb T_2.
\end{aligned}
\end{equation}
%--------------------------------------------------
Now, apply \eqref{PD} to $\bar u(t^{(n+1)})$ and to $\bar v(t^{(n+1)})$ to obtain, thanks to \eqref{hd-1}
\begin{equation}\label{eq-proof-Pi-u}
\begin{aligned}
\Big\| \bar u^{(n+1)}&-\Pi_\disc  P_\disc \bar u(t^{(n+1)}) \Big\|_{L^2(\O)}\\
&\leq \Big\| \bar u^{(n+1)}-\bar u(t^{(n+1)}) \Big\|_{L^2(\O)} + S_\disc(\bar u(t^{(n+1)}))\\
&\leq C_{13}(\delta t+h_\disc),\mbox{ and }
\end{aligned}
\end{equation}
%-------------------------------------
\begin{equation}\label{eq-proof-Pi-v}
\begin{aligned}
\Big\| \bar v^{(n+1)}&-\Pi_\disc  P_\disc \bar v(t^{(n+1)}) \Big\|_{L^2(\O)}\\
&\leq \Big\| \bar v^{(n+1)}-\bar v(t^{(n+1)}) \Big\|_{L^2(\O)} + S_\disc(\bar v(t^{(n+1)}))\\
&\leq C_{14}(\delta t+h_\disc).
\end{aligned}
\end{equation}
%--------------------------------------
By Introducing the term $\Pi_\disc P_\disc \bar u(t^{(n+1)})$, using \eqref{eq-proof-Pi-u} and applying the Young's inequality with a small positive parameter $\varepsilon_3$, one can have
\begin{equation}\label{eq-proof-T1}
\begin{aligned}
\cT_1 \leq &\dsp\sum_{n=0}^{m-1} \delta t^{(n+\frac{1}{2})} L\|\Pi_\disc E^{(n+1)}\|_{L^2(\O)}\times\\
&\Big(
\|\bar u^{(n+1)}-\Pi_\disc P_\disc \bar u(t^{(n+1)})\|_{L^2(\O)}
+\|\Pi_\disc P_\disc \bar u(t^{(n+1)})-\Pi_\disc u^{(n+1)}\|_{L^2(\O)}
\Big)\\
&\leq \dsp\sum_{n=0}^{m-1} \delta t^{(n+\frac{1}{2})} L\|\Pi_\disc E^{(n+1)}\|_{L^2(\O)}
\Big(
C_{13}(\delta t + h_\disc)
+\|\Pi_\disc E^{(n+1)}\|_{L^2(\O)}
\Big)\\
&\leq \frac{L(2+\varepsilon_3)}{2}\dsp\sum_{n=0}^{m-1} \delta t^{(n+\frac{1}{2})}\|\Pi_\disc E^{(n+1)}\|_{L^2(\O)}^2
+\frac{1}{2\varepsilon_3}\dsp\sum_{n=0}^{m-1} \delta t^{(n+\frac{1}{2})}C_{13}(\delta t + h_\disc)^2.
\end{aligned}
\end{equation}
%---------------------------------
Employing computations similar to those above, we infer, thanks to \eqref{eq-proof-Pi-v}, where $\varepsilon_4,\varepsilon_5$ are small positive parameters linked to the Young's inequality,
\begin{equation}\label{eq-proof-T2}
\begin{aligned}
\cT_2 \leq &\dsp\sum_{n=0}^{m-1} \delta t^{(n+\frac{1}{2})} L\|\Pi_\disc E^{(n+1)}\|_{L^2(\O)}\times\\
&\Big(
\|\bar v^{(n+1)}-\Pi_\disc P_\disc \bar v(t^{(n+1)})\|_{L^2(\O)}
+\|\Pi_\disc P_\disc \bar v(t^{(n+1)})-\Pi_\disc v^{(n+1)}\|_{L^2(\O)}
\Big)\\
&\leq \dsp\sum_{n=0}^{m-1} \delta t^{(n+\frac{1}{2})} L\|\Pi_\disc E^{(n+1)}\|_{L^2(\O)}
\Big(
C_{14}(\delta t + h_\disc)
+\|\Pi_\disc \widetilde E^{(n+1)}\|_{L^2(\O)}
\Big)\\
&\leq \frac{L(\varepsilon_4+\varepsilon_5)}{2}\dsp\sum_{n=0}^{m-1} \delta t^{(n+\frac{1}{2})}\|\Pi_\disc E^{(n+1)}\|_{L^2(\O)}^2
+\frac{1}{2\varepsilon_4}\dsp\sum_{n=0}^{m-1} \delta t^{(n+\frac{1}{2})}C_{14}(\delta t + h_\disc)^2\\
&\quad+\frac{1}{2\varepsilon_5}\dsp\sum_{n=0}^{m-1}\delta t^{(n+\frac{1}{2})}\|\Pi_\disc \widetilde E^{(n+1)}\|_{L^2(\O)}^2.
\end{aligned}
\end{equation}
%---------------------------------
Collecting \eqref{eq-proof-T1} and \eqref{eq-proof-T2} in \eqref{eq-proof-F}, we obtain
\begin{equation}\label{eq-proof-F-1}
\begin{aligned}
&\dsp\sum_{n=0}^{m-1} \delta t^{(n+\frac{1}{2})}\int_\O \Pi_\disc E^{(n+1)}(\x)\Big(F_1(\bar u^{(n+1)},\bar v^{(n+1)})-F_1(\Pi_\disc u^{(n+1)},\Pi_\disc v^{(n+1)}) \Big) \ud \x\\
&\leq\frac{L(2+\varepsilon_3+\varepsilon_4+\varepsilon_5)}{2}\dsp\sum_{n=0}^{m-1} \delta t^{(n+\frac{1}{2})}\|\Pi_\disc  E^{(n+1)}\|_{L^2(\O)}^2\\
&\quad+\frac{\varepsilon_3+\varepsilon_4}{2\varepsilon_3 \varepsilon_4}\dsp\sum_{n=0}^{m-1} \delta t^{(n+\frac{1}{2})}C_{15}(\delta t + h_\disc)^2
+\frac{1}{2\varepsilon_5}\dsp\sum_{n=0}^{m-1} \delta t^{(n+\frac{1}{2})}\|\Pi_\disc \widetilde E^{(n+1)}\|_{L^2(\O)}^2.
\end{aligned}
\end{equation}
%---------------------------------
Similarly, we can address the third term on the right-hand side of Inequality \eqref{eq-proof-8-v}. We determine that $\varepsilon_6, \varepsilon_7, \varepsilon_8$ represent small positive parameters associated with the Young's inequality 
\begin{equation}\label{eq-proof-G-1}
\begin{aligned}
&\dsp\sum_{n=0}^{m-1} \delta t^{(n+\frac{1}{2})}\int_\O \Pi_\disc \widetilde E^{(n+1)}(\x)\Big(F_2(\bar u^{(n+1)},\bar v^{(n+1)})-F_2(\Pi_\disc u^{(n+1)},\Pi_\disc v^{(n+1)}) \Big) \ud \x\\
&\leq\frac{L(2+\varepsilon_6+\varepsilon_7+\varepsilon_8)}{2}\dsp\sum_{n=0}^{m-1} \delta t^{(n+\frac{1}{2})}\|\Pi_\disc \widetilde E^{(n+1)}\|_{L^2(\O)}^2\\
&\quad+\frac{\varepsilon_6+\varepsilon_7}{2\varepsilon_6 \varepsilon_7}\dsp\sum_{n=0}^{m-1} \delta t^{(n+\frac{1}{2})}C_{16}(\delta t + h_\disc)^2
+\frac{1}{2\varepsilon_8}\dsp\sum_{n=0}^{m-1} \delta t^{(n+\frac{1}{2})}\|\Pi_\disc  E^{(n+1)}\|_{L^2(\O)}^2.
\end{aligned}
\end{equation}
%----------------------------------------
Therefore, collecting in \eqref{eq-proof-8-u}, the bounds \eqref{eq-proof-9-u} and \eqref{eq-proof-F-1}, and using $\sum_{n=0}^{m-1}\delta t^{(n+\frac{1}{2})}\leq T$, we obtain
\begin{equation}\label{new-E-u}
\begin{aligned}
&\frac{1}{2}\dsp\int_\O (\Pi_\disc  E^{(m)}(\x))^2 \ud \x
+\frac{\underline\lambda_1-\varepsilon_1}{2}\dsp\sum_{n=0}^{m-1} \delta t^{(n+\frac{1}{2})} \left\| \nabla_\disc  E^{(n+1)} \right\|_{L^2(\O)^d}^2 \\
&\leq\frac{L(2+\varepsilon_3+\varepsilon_4+\varepsilon_5)}{2}\dsp\sum_{n=0}^{m-1} \delta t^{(n+\frac{1}{2})}\|\Pi_\disc  E^{(n+1)}\|_{L^2(\O)}^2\\
&\quad+\frac{1}{2\varepsilon_5}\dsp\sum_{n=0}^{m-1} \delta t^{(n+\frac{1}{2})} \|\Pi_\disc \widetilde E^{(n+1)}\|_{L^2(\O)}^2\\
&\quad+(C_{11}h_\disc+ E_\disc^0)^2
+T(\frac{1}{2\varepsilon_1}+\frac{\varepsilon_3+\varepsilon_4}{2\varepsilon_3 \varepsilon_4})C_{17}(\delta t+h_\disc)^2.
\end{aligned}
\end{equation}
%--------------------------------
Additionally, assembling in \eqref{eq-proof-8-v}, the bounds \eqref{eq-proof-9-v} and \eqref{eq-proof-G-1} give
\begin{equation}\label{new-E-v}
\begin{aligned}
&\frac{1}{2}\dsp\int_\O (\Pi_\disc \widetilde E^{(m)}(\x))^2 \ud \x
+\frac{\underline\lambda_2-\varepsilon_2}{2}\dsp\sum_{n=0}^{m-1} \delta t^{(n+\frac{1}{2})} \left\| \nabla_\disc \widetilde E^{(n+1)} \right\|_{L^2(\O)^d}^2 \\
&\leq
\dsp\frac{L(2+\varepsilon_6+\varepsilon_7+\varepsilon_8)}{2}\sum_{n=0}^{m-1} \delta t^{(n+\frac{1}{2})} \|\Pi_\disc  \widetilde E^{(n+1)}\|_{L^2(\O)}^2\\
&\quad+\frac{1}{2\varepsilon_8}\dsp\sum_{n=0}^{m-1} \delta t^{(n+\frac{1}{2})} \|\Pi_\disc  E^{(n+1)}\|_{L^2(\O)}^2\\
&\quad+(C_{12}h_\disc+ \widetilde E_\disc^0)^2
+T(\frac{1}{2\varepsilon_2}+\frac{\varepsilon_6+\varepsilon_7}{2\varepsilon_6 \varepsilon_7})C_{18}(\delta t+h_\disc)^2.
\end{aligned}
\end{equation}
%-------------------------------------
The direct application of the discrete Gronwall’s Lemma \cite[Lemma 10.5]{26} to the inequalities \eqref{new-E-u} and \eqref{new-E-v} yields
\begin{equation}\label{eq-proof-10-u}
\begin{aligned}
&\frac{1}{2}\dsp\int_\O(\Pi_\disc  E^{(m)}(\x))^2 \ud \x
+\frac{\underline\lambda_1-\varepsilon_1}{2}\dsp\sum_{n=0}^{m-1} \delta t^{(n+\frac{1}{2})} \left\| \nabla_\disc  E^{(n+1)} \right\|_{L^2(\O)^d}^2 \\
&\leq \exp(TM_1)\Big(
\dsp\frac{1}{2\varepsilon_5}\sum_{n=0}^{m-1} \delta t^{(n+\frac{1}{2})}\|\Pi_\disc \widetilde E^{(n+1)}\|_{L^2(\O)}^2\\
&\quad+(C_{11}h_\disc+ E_\disc^0)^2
+T(\frac{1}{2\varepsilon_1}+\frac{\varepsilon_3+\varepsilon_4}{2\varepsilon_3 \varepsilon_4})C_{17}(\delta t+h_\disc)^2
\Big),\mbox{ and }
\end{aligned}
\end{equation}
%-------------------
\begin{equation}\label{eq-proof-10-v}
\begin{aligned}
&\frac{1}{2}\dsp\int_\O(\Pi_\disc \widetilde E^{(m)}(\x))^2 \ud \x
+\frac{\underline\lambda_2-\varepsilon_2}{2}\dsp\sum_{n=0}^{m-1} \delta t^{(n+\frac{1}{2})} \left\| \nabla_\disc \widetilde E^{(n+1)} \right\|_{L^2(\O)^d}^2 \\
&\leq \exp(TM_2)\Big(
\dsp\frac{1}{2\varepsilon_8}\sum_{n=0}^{m-1} \delta t^{(n+\frac{1}{2})}\|\Pi_\disc  E^{(n+1)}\|_{L^2(\O)}^2\\
&\quad+(C_{12}h_\disc+ \widetilde E_\disc^0)^2
+T(\frac{1}{2\varepsilon_2}+\frac{\varepsilon_6+\varepsilon_7}{2\varepsilon_6 \varepsilon_7})C_{18}(\delta t+h_\disc)^2
\Big),
\end{aligned}
\end{equation}
%---------------------------
where $M_1:=\frac{L(2+\varepsilon_3+\varepsilon_4+\varepsilon_5)}{2}$ and $M_2:=\frac{L(2+\varepsilon_6+\varepsilon_7+\varepsilon_8)}{2}$. Since we can simplify the terms on the right-hand sides with those on the left-hand sides in both relations, we can express this due to the discrete Poincaré inequality \eqref{ponc-enq}
\begin{equation}\label{eq-proof-12-1}
\begin{aligned}
&\Big(\frac{\underline\lambda_1-\varepsilon_1}{2}-\frac{C_\disc}{2\varepsilon_8}\exp(TM_2)\Big)\dsp\sum_{n=0}^{m-1} \delta t^{(n+\frac{1}{2})} \left\| \nabla_\disc  E^{(n+1)} \right\|_{L^2(\O)^d}^2\\
&\quad+\Big(\frac{\underline\lambda_2-\varepsilon_2}{2}-\frac{C_\disc}{2\varepsilon_5}\exp(TM_1)\Big)\dsp\sum_{n=0}^{m-1} \delta t^{(n+\frac{1}{2})} \left\| \nabla_\disc \widetilde E^{(n+1)} \right\|_{L^2(\O)^d}^2\\
&\leq C_{19}(\delta t+h_\disc)^2 
+C_{20}(h_\disc+ E_\disc^0)^2
+C_{21}(h_\disc+ \widetilde E_\disc^0)^2:=\mathbb T_3\mbox{ and }
\end{aligned}
\end{equation}
%--------------------------------
\begin{equation}\label{eq-proof-12-2}
\begin{aligned}
\frac{1}{2}\|\Pi_\disc  E^{(m)}\|_{L^2(\O)}^2
+\frac{1}{2}\|\Pi_\disc  \widetilde E^{(m)}\|_{L^2(\O)}^2
\leq \mathbb T_3.
\end{aligned}
\end{equation}
%--------------------------------------
From the triangle inequality, \eqref{new-eq-proof-100}, and \eqref{eq-proof-12-2} combined with the power-of-sums inequality $(\alpha+\beta)^{1/2}\leq \alpha^{1/2}+\beta^{1/2}$, we attain for all $m\in\{1,...,N\}$, thanks to the relation \eqref{hd-1}
\begin{equation}\label{eq-proof-11-u}
\begin{aligned}
\|\Pi_\disc u^{(m)}-\bar u(t^{(m)})\|_{L^2(\O)}
&\leq \|\Pi_\disc E^{(m)}\|_{L^2(\O)}
+\|\Pi_\disc P_\disc\bar u(t^{(m)})-\bar u(t^{(m)})\|_{L^2(\O)}\\
&\leq 
\mathbb T_3+\sqrt{2}S_\disc(\bar u(t^{(m)}))\\
&\leq C_{22}(\delta t+h_\disc+E_\disc^0 + \widetilde E_\disc^0),\mbox{ and } 
\end{aligned}
\end{equation}
%-----------------------------
\begin{equation}\label{eq-proof-11-v}
\begin{aligned}
\|\Pi_\disc v^{(m)}-\bar v(t^{(m)})\|_{L^2(\O)}
&\leq \|\Pi_\disc \widetilde E^{(m)}\|_{L^2(\O)}
+\|\Pi_\disc P_\disc\bar v(t^{(m)})-\bar v(t^{(m)})\|_{L^2(\O)}\\
&\leq 
\mathbb T_3+\sqrt{2}S_\disc(\bar v(t^{(m)}))\\
&\leq C_{22}(\delta t+h_\disc+E_\disc^0 + \widetilde E_\disc^0)
\end{aligned}
\end{equation}
%-------------------------------------
Whereas using again triangle the inequality, \eqref{new-eq-proof-100}, \eqref{eq-proof-12-1} with $m=N-1$, and the power-of-sums inequality $(\alpha+\beta)^{2}\leq 2\alpha^2+2\beta^2$ leads to owing to the relation \eqref{hd-1} 
\begin{equation}\label{eq-proof-13-u}
\begin{aligned}
&\dsp\sum_{n=0}^{N-1}\delta t^{(n+\frac{1}{2})}\|\nabla_\disc u^{(n+1)} - \nabla\bar u(t^{(n+1)})\|_{L^2(\O)^d}\\
&\leq \dsp\sum_{n=0}^{N-1}\delta t^{(n+\frac{1}{2})}\|\nabla_\disc E^{(n+1)}\|_{L^2(\O)^d}
+\dsp\sum_{n=0}^{N-1}\delta t^{(n+\frac{1}{2})}\|\nabla_\disc P_\disc\bar u(t^{(n+1)})-\nabla\bar u(t^{(n+1)})\|_{L^2(\O)^d}\\
&\leq 
C_{23}\mathbb T_3+\dsp\sum_{n=0}^{N-1}\delta t^{(n+\frac{1}{2})}S_\disc(\bar u(t^{(n+1)}))\\
&\leq C_{24}(\delta t+h_\disc+E_\disc^0 + \widetilde E_\disc^0),\mbox{ and }
\end{aligned}
\end{equation}
%-------------------------------------
\begin{equation}\label{eq-proof-13-v}
\begin{aligned}
&\dsp\sum_{n=0}^{N-1}\delta t^{(n+\frac{1}{2})}\|\nabla_\disc v^{(n+1)} - \nabla\bar v(t^{(n+1)})\|_{L^2(\O)^d}\\
&\leq \dsp\sum_{n=0}^{N-1}\delta t^{(n+\frac{1}{2})}\|\nabla_\disc \widetilde E^{(n+1)}\|_{L^2(\O)^d}
+\dsp\sum_{n=0}^{N-1}\delta t^{(n+\frac{1}{2})}\|\nabla_\disc P_\disc\bar v(t^{(n+1)})-\nabla\bar v(t^{(n+1)})\|_{L^2(\O)^d}\\
&\leq 
C_{25}\mathbb T_3+\dsp\sum_{n=0}^{N-1}\delta t^{(n+\frac{1}{2})}S_\disc(\bar v(t^{(n+1)}))\\
&\leq C_{26}(\delta t+h_\disc+E_\disc^0 + \widetilde E_\disc^0).
\end{aligned}
\end{equation}
%-------------------------
The conclusion is derived from the tringular inequality, \eqref{eq-proof-11-u}, \eqref{eq-proof-11-v}, \eqref{eq-proof-13-u}, \eqref{eq-proof-13-v}, and the Lipschtiz-continuity of the solutions $\bar u,\bar v:[0,T]\to H^1(\O)$. 
\end{proof}
%---------------------------------

\begin{remark}
Note that Estimates \eqref{eq-error} provide convergence rates for both conforming and non-conforming methods, while the estimates in \cite{7,15} seem to apply solely to a specific scheme and particular choices of reaction functions. Furthermore, it appears that we initially provided a general error estimate for approximating the general reaction-diffusion system.

From Estimates \eqref{eq-error}, we can conclude the order of convergence in terms of time and mesh sizes owing to the connection between the space size $h_\disc$ and the mesh size ($h_\mesh$) for mesh-based gradient discretisations. For additional details, refer to \cite[Remark 2.24]{21}.

For parabolic problems, obtaining error estimates typically requires additional regularity on solutions, such as the solution belongs to the space $W^{1,\infty}(0,T;W^{2,\infty}(\O))$ or $C^2([0,T];H^2(\O))$, see \cite{22,24,29} for instance. 

%Here, if $\Lambda_1$ and $\Lambda_2$ are Lipschitz continuous, the assumption $(\div(\Lambda_1\nabla\bar u),\div(\Lambda_2\nabla\bar v)) \in (L^2(0,T;L^2(\O)^d))^2$ is justifiable due to the $H^2$-regularity result on the solutions $\bar u$ and $\bar v$. 

\end{remark}

%------------------------------------------
%------------------------------------------
%SECTION
\section{Numerical Example}\label{sec-num}
To illustrate our convergence rates, we introduce here a numerical example based on a polytopal method called the hybrid mimetic mixed (HMM) method \cite{25}. We point out that the gradient scheme \eqref{rm-disc-pblm} with a particular choice of the space-time gradient discretisation yields the HMM scheme for the problem \eqref{rm-weak}, see \cite[Section 5]{6} for a full presentation of the HMM method for SRDM. For all standard low-order gradient discretisations (which include the HMM method), Theorem \ref{thm-err-rm} provides an $\mathcal O(h_\mesh)$ rate of convergence for the functions and the gradients.  

We consider here \cite[Example 1]{14}. Letting $\O=(0,1)^2$ and $T=1$, we solve Problem \eqref{eq-rm} with tensors $\Lambda_1=\mathbb A \mathbb D_1 \mathbb A^{T\rm }$ and $\Lambda_2=\mathbb B \mathbb D_2 \mathbb B^{T\rm }$, where
%--------------------
\[
\mathbb A=
\begin{pmatrix}
  \cos \theta_1 & -\sin \theta_1\\
  \sin\theta_1 & \cos\theta_1 
 \end{pmatrix},\quad
 \mathbb B=
\begin{pmatrix}
  \cos \theta_2 & -\sin \theta_2\\
  \sin\theta_2 & \cos\theta_2 
 \end{pmatrix},\quad
\]
and
\[
\mathbb D_1=
\begin{pmatrix}
  g_1 & 0\\
  0 & g_2 
 \end{pmatrix},\quad
 \mathbb D_2=
\begin{pmatrix}
  g_2 & 0\\
  0 & g_1 
 \end{pmatrix}.
\]
%----------------------------
We take $g_1(x,y)=1+2x^2+y^2$, $g_2(x,y)=1+x^2+2y^2$, $\theta_1=\frac{5\pi}{12}$ and $\theta_2=\frac{\pi}{3}$. The reaction terms are defined by
\begin{equation*}
\begin{aligned}
F_1(\bar u,\bar v)&=\bar u(1-\bar u)(\bar u-0.1),\\ 
F_2(\bar u,\bar v)&=\bar u-\bar v.
\end{aligned}
\end{equation*}
%-----------
The exact solution is given by
\begin{equation*}
\begin{aligned}
u(x,y,t)&=e^{-t}\sin(\pi x) \sin(\pi y),\\
v(x,y,t)&=e^{-t}\cos(\pi x) \cos(\pi y).
\end{aligned}
\end{equation*}
%----------------------------------------
Note that the problem solved here is subjected to the non-homogenous Dirichlet boundary conditions and additional source functions on the right-hand side that can be attained from the exact solution.

\begin{figure}[ht]\label{fig-mesh}
	\begin{center}
	\begin{tabular}{cc}
	\includegraphics[width=0.40\linewidth]{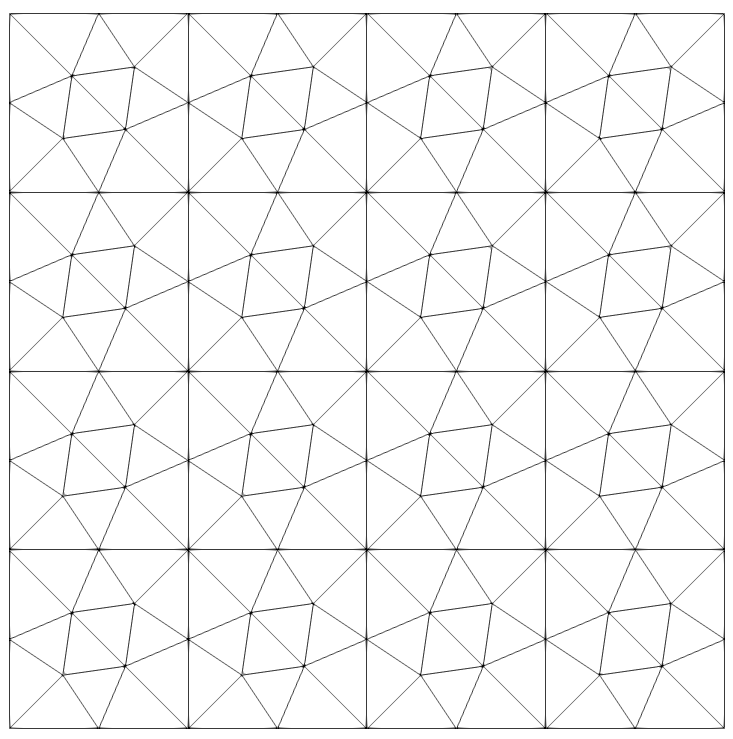} & \includegraphics[width=0.40\linewidth]{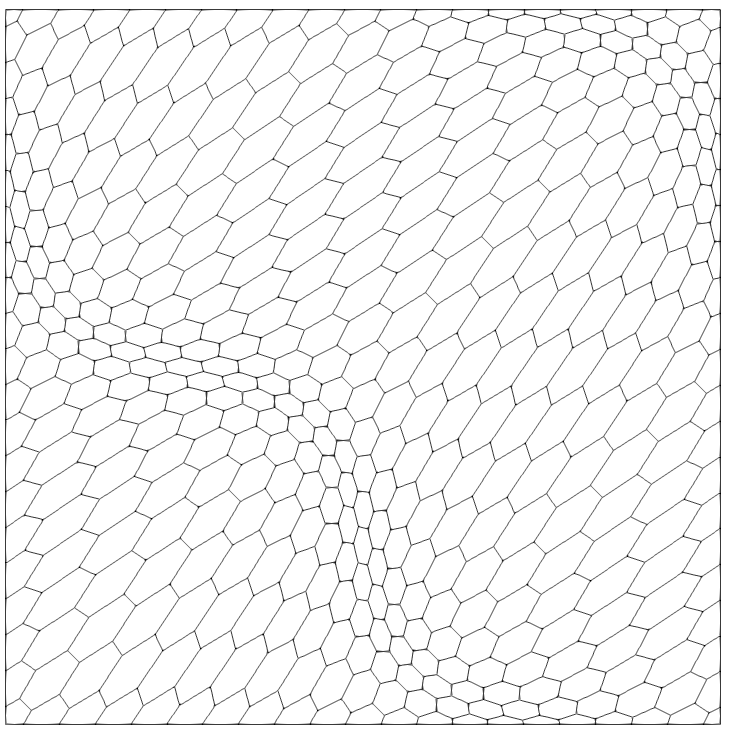}\\
	 \texttt{Triangular}  & \texttt{Hexagonal}
	\end{tabular}
	\end{center}
	\caption{Sample of the polygonal meshes.}
	\label{fig-1}
\end{figure}
%-------------------------------------------------

For our simulation, we use two distinct families of polygonal meshes as in Figure \ref{fig-1}. The relative errors for $\bar u$, $\bar v$, $\nabla\bar u$, and $\nabla\bar v$ are summarised in Table \ref{tab-1}, along with their corresponding order of convergence based on the mesh size $h$. The results demonstrate that the HMM scheme achieves super-convergence order for the $L^2$ errors of the solutions. As a polytopal scheme, it performs effectively on hexagonal meshes. 

Figures \ref{fig-1}, \ref{fig-2} present convergence graphs corresponding to the solutions and their gradients obtained with the use of hexagonal mesh. We observe that the errors reduce linearly (in fact, greater than first-order) with respect to the mesh size $h$, which aligns with the expected rate of convergence in spatial space outlined in Theorem \ref{thm-err-rm}.
%-----------------------------------------

\begin{table}[]
\begin{tabular}{c c c c c}
\hline
$h_\mesh$&
%------------------
$0.125000000$&
$0.062500000$&
$0.031250000$&
$0.015625000$  
\\ \hline
%------------------
$\frac{\| \bar u(\cdot,t^{(n)}) - \Pi_\disc u^{(n)}\|_{L^{2}(\O)}}{\| \bar u(\cdot,t^{(n)})\|_{L^{2}(\O)}}$&
$0.004738288$&
$0.0013162825$&
$0.0004082468$&
$0.0001064975$\\ 
%------------------
Rate&
$1.85$&
$1.69$&
1.94&
--
\\ 
%------------------
$\frac{\| \bar v(\cdot,t^{(n)}) - \Pi_\disc v^{(n)}\|_{L^{2}(\O)}}{\| \bar v(\cdot,t^{(n)})\|_{L^{2}(\O)}}$&
$0.005458805$&
$0.001343458$&
$0.000356627$&
$0.000105361$\\ 
%------------------
Rate&
$2.02$&
$1.91$&
$1.76$&
--
\\
%-----------------------
$\frac{\| \nabla\bar u(\cdot,t^{(n)}) - \nabla_\disc u^{(n)}\|_{L^{2}(\O^2)}}{\| \nabla\bar u(\cdot,t^{(n)})\|_{L^{2}(\O)^2}}$&
$0.069660264$&
$0.035359614$&
$0.0176641949$&
$0.0088303991$\\ 
%--------------------
Rate&
$0.98$&
1.00&
1.00&
--
\\
%------------------------
$\frac{\| \nabla\bar v(\cdot,t^{(n)}) - \nabla_\disc v^{(n)}\|_{L^{2}(\O)^2}}{\| \nabla\bar v(\cdot,t^{(n)})\|_{L^{2}(\O)^2}}$&
$0.076321506$&
$0.035358391$&
$0.0176643386$&
$0.008830754$\\ 

%--------------------
Rate&
$1.11$&
$1.00$&
$1.00$&
--
\\ 
%---------------------
\hline
\end{tabular}
\caption{The errors on triangular meshes.}
\label{tab-1}   
\end{table}
%-----------------------------------
%-----------------------------------

\begin{figure}[ht]
	\begin{center}
	\includegraphics[scale=0.4]{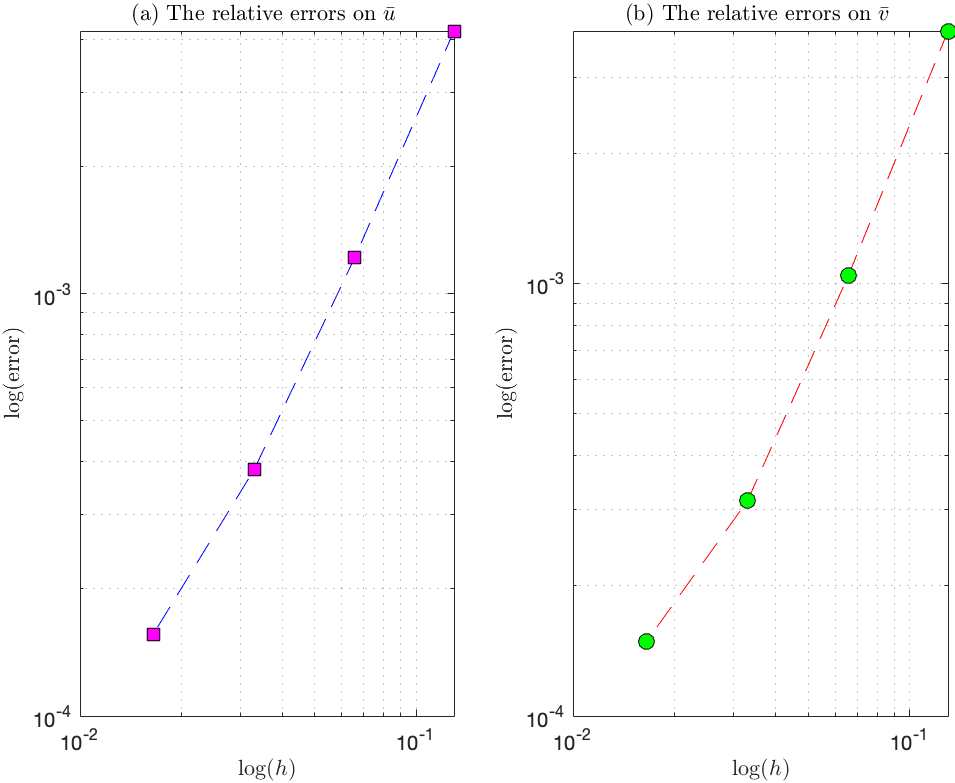}
	\end{center}
\caption{The errors on hexagonal meshes.}	
\label{fig-1}
\end{figure}
%--------------------------

\begin{figure}[ht]
	\begin{center}
	\includegraphics[scale=0.4]{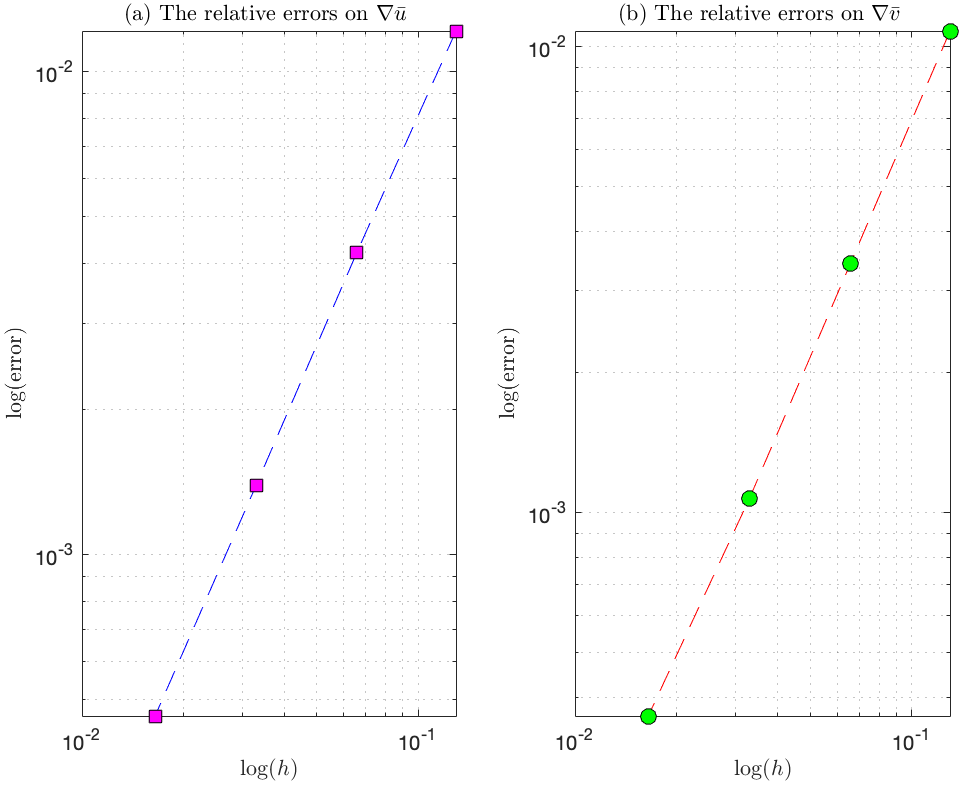}
	\end{center}
\caption{The errors on hexagonal meshes.}	
\label{fig-2}
\end{figure}
%------------------------------------------------

% --- BIBLIOGRAPHY ---

\bibliographystyle{siam}
\bibliography{RDM-ref}

%-------------------------------------------------

\section*{Declarations}
\begin{itemize}
\item Funding: Funding information is not applicable.
\item Conflict of interest: The author declares that he has no competing interests.
\item Data availability: The manuscript has no associated data. 
\item Author contribution: The paper has only one author.
\end{itemize}

\end{document}